\documentclass[11pt,a4paper]{article}
\usepackage{a4wide}
\usepackage{amssymb,amsfonts,amsmath,amsthm}

\usepackage{graphicx}
\usepackage[hidelinks]{hyperref}
\usepackage{mathtools}
\usepackage{xparse}
\usepackage{amssymb,amsfonts,amsmath}
\usepackage{graphicx,color,changebar}

\newcommand{\eg}{e.\,g.,\ }

\newcommand{\T}{\mathcal{T}}

\newcommand{\Hn}{ {\mathbb{H}_n} }   
\newcommand{\XWpd}{ {\mathbb{X}^{\raisebox{0.2em}{{\fontsize{3}{2}\selectfont $>$}}}} }
\newcommand{\XWpdpd}{ {\mathbb{X}^{\raisebox{0.2em}{{\fontsize{3}{2}\selectfont $\gg$}}}} }  
\usepackage{todonotes}

\newtheorem{theorem}{Theorem}[section]
\newtheorem{remark}[theorem]{Remark}

\newtheorem{lemma}[theorem]{Lemma}
\newtheorem{definition}[theorem]{Definition}
\newtheorem{corollary}[theorem]{Corollary}
\newcommand{\M}{{\mathcal M}}

\renewcommand{\S}{{\mathcal S}}

\DeclareMathOperator{\diag}{diag}

\DeclareMathOperator{\rank}{rank}
\DeclareMathOperator{\trace}{trace}
\newcommand {\matr}      [1] {\left[\begin{array}{#1}}
\newcommand {\rix}          {\end{array}\right]}

\newcommand{\ie}{i.\,e.\ }

\usepackage{xstring}
\newcommand{\mat}[3][C]{
	\mathbb{#1}^{
		\IfSubStr{#2}{+}{(#2)}{#2}
		\times
		\IfSubStr{#3}{+}{(#3)}{#3}
	}
}
\DeclareDocumentCommand{\matz}{m m O{K} O{z}}{\mathbb{#3}[{#4}]^
	{
		\IfSubStr{#1}{+}{(#1)}{#1}
		\times
		\IfSubStr{#2}{+}{(#2)}{#2}
	}}

\title{Optimal robustness of passive discrete time systems}  

\author{V. Mehrmann\footnotemark[1], P. Van Dooren\footnotemark[2]
}
\date{\today}
\begin{document}
\maketitle

\begin{abstract} We construct optimally robust realizations of a given rational transfer function that represents a passive discrete-time system.
We link it to the solution set of linear matrix inequalities defining passive transfer functions.
We also consider the problem of finding the nearest passive system to a given non-passive one.
\end{abstract}

\footnotetext[1]{
Institut f\"ur Mathematik MA 4-5, TU Berlin, Str.\ des 17.\ Juni 136,
D-10623 Berlin, Germany.
\url{mehrmann@math.tu-berlin.de}.
}
\footnotetext[2]{
Department of Mathematical Engineering, Universit\'e catholique de Louvain, Louvain-La-Neuve, Belgium.
\url{paul.vandooren@uclouvain.be}.
}

\section{Introduction}
We consider realizations of linear discrete-time dynamical systems for which the associated transfer function is {\em passive}. Such transfer functions
play a fundamental role in systems and control theory: they represent \eg spectral density functions of stochastic processes, show up in spectral factorizations and
are also related to discrete-time algebraic Riccati equations. Passive transfer functions can be described using convex sets,
and this property has lead to the extensive use of convex optimization techniques in this area \cite{BoyEFB94}.

In this paper we show that in the set of possible realizations of a given passive transfer function, there is a subset that maximizes robustness, in the sense
that their so-called {\em passivity radius} is nearly optimal. Related results for continuous-time systems were already obtained in a companion paper \cite{MeV19}.
Here we consider the discrete-time system
\begin{equation} \label{statespace_d}
 \begin{array}{rcl} x_{k+1} & = & Ax_k + B u_k,\ x_0=0,\\
y_k&=& Cx_k+Du_k,
\end{array}
\end{equation}
where $u_k \in \mathbb{C}^m$, $x_k \in \mathbb{C}^n$, and  $y_k \in \mathbb{C}^m$  are vector-valued sequences denoting, respectively, the \emph{input}, \emph{state},
and \emph{output} of the system. Denoting real and complex $n$-vectors ($n\times m$ matrices) by $\mathbb R^n$, $\mathbb C^{n}$ ($\mathbb R^{n \times m}$, $\mathbb{C}^{n \times m}$), respectively, the coefficient matrices satisfy $A\in \mathbb{C}^{n \times n}$,   $B\in \mathbb{C}^{n \times m}$, $C\in \mathbb{C}^{m \times n}$, and  $D\in \mathbb{C}^{m \times m}$.

We restrict ourselves to systems which are \emph{minimal}, \ie the pair $(A,B)$ is \emph{controllable}  (for all $z\in \mathbb C$, $\rank \mbox{\small $[\,z I-A \ B\,]$} =n$), and the pair $(A,C)$ is \emph{observable} (\ie $(A^\mathsf{H},C^\mathsf{H})$ is controllable). Here, the Hermitian (or conjugate) transpose (transpose) of a vector or matrix $V$ is denoted by
$V^{\mathsf{H}}$ ($V^{\mathsf{T}}$) and the identity matrix is denoted by $I_n$ or $I$ if the dimension is clear. We furthermore require that input and output dimensions are equal to $m$.

\emph{Passive} systems are well studied in the continuous-time case, starting with the works \cite{Wil72a,Wil72b}. Here we consider
the equivalent definition in the discrete-time case and derive so-called {\em normalized passive realizations} that could be considered as ``discrete-time port-Hamiltonian systems''.
Similar attempts were already made in the literature \cite{KoL18},\cite{SSV12},\cite{TCV18}.

The paper is organized as follows. After going over some preliminaries in Section~\ref{sec:prelim}, we characterize in Section~\ref{sec:norm} what we called normalized passive realizations
of a discrete-time passive system. We then show in Section~\ref{sec:passrad} their relevance in estimating the passivity radius of sicrete-time passive systems and construct in
Section~\ref{sec:maxpass} realizations with nearly optimal robustness margin for passivity.
In Section \ref{sec:computing} we describe an algorithm to compute this robustness margin.
In Section \ref{sec:distance} we show how to use these ideas to estimate the distance to the set of discrete-time passive systems.

\section{Passive systems} \label{sec:prelim}
Throughout this article we will use the following notation.
We denote the set of Hermitian matrices in $\mathbb{C}^{n \times n}$ by $\Hn$.
Positive definiteness (semi-definiteness) of  $A\in \Hn$ is denoted by $A>0$ ($A\geq 0$).
The real and imaginary parts of a complex matrix $Z$ are written as $\Re (Z)$ and $\Im (Z)$, respectively, and $\imath$ is the imaginary unit.
We consider functions over $\Hn$, which is a vector space if considered as a \emph{real} subspace of $\mat[R]nn+\imath \mat[R]nn$.

The concept of \emph{passivity} is well studied. We briefly recall some important properties following \cite{Wil72b},
and refer to the literature for proofs and for a more detailed survey. Consider a discrete-time system (\ref{statespace_d}) with minimal state-space model
$$\M:=\{A,B,C,D\}$$
and transfer function $\mathcal T(z):=C(zI_n-A)^{-1}B+D$ and define the complex analytic function of $z\in \mathbb{C}$~:
\[ \Phi(z):=\mathcal T^{\mathsf{H}}(z^{-1}) + \mathcal T(z), \] which coincides with the Hermitian part of $\mathcal T(z)$ on the  unit circle:
\[ \Phi(e^{\imath \omega})=[\mathcal T(e^{\imath \omega})]^\mathsf{H} + \mathcal T(e^{\imath \omega}). \]

The transfer function $\mathcal T(z)$ is called {\em strictly positive-real} \/if $\Phi(e^{\imath \omega})> 0 $ for all $\ \omega\in [-\pi,\pi]$ and
it is called \emph{positive-real} if $\Phi(e^{\imath \omega})\geq 0 $ for all $\ \omega\in [-\pi,\pi]$;
$\mathcal T(z)$ is called {\em asymptotically stable} if the eigenvalues of $A$ are in the open unit disc, and it is called {\em stable} if the
eigenvalues of $A$ are in the closed unit disc, with any eigenvalues occurring on the unit circle being semi-simple.
With these two properties, then $\mathcal T(z)$ is called {\em strictly passive} \/if  it is strictly positive-real and asymptotically stable and it is
called \emph{passive} if it is positive real and stable.

The transfer function $\mathcal T(z)$ is the Schur complement of the so-called {\em system pencil}
\begin{equation} \label{statespace}
S(z) :=
\left[ \begin{array}{cc|c} 0 & A-zI_n & B \\
zA^{\mathsf{H}}-I_n & 0 & C^{\mathsf{H}} \\ \hline zB^{\mathsf{H}} & C & D^{\mathsf{H}}+D  \end{array} \right]
\end{equation}
and if the model $\M$ is minimal, then the finite generalized eigenvalues of $S(z)$ are the finite zeros of $\Phi(z)$.
The following equivalence transformation, using an arbitrary matrix $X \in \Hn$, leaves the Schur complement, and hence also the transfer function $\Phi(z)$, unchanged
{\small \begin{equation} \label{Sz}
\left[ \begin{array}{cc|c} 0 & A-zI_n & B \\
zA^{\mathsf{H}}-I_n & X - A^{\mathsf{H}}XA & C^{\mathsf{H}} - A^{\mathsf{H}}XB \\ \hline zB^{\mathsf{H}} &
C- B^{\mathsf{H}}XA & D^{\mathsf{H}}+D -B^{\mathsf{H}}XB \end{array} \right] =
\left[ \begin{array}{cc|c} I_n & 0 & 0 \\
-A^\mathsf{H}X & I_n & 0 \\ \hline -B^\mathsf{H}X & 0 & I_m  \end{array} \right] S(z)
\left[ \begin{array}{cc|c} I_n & -X & 0 \\
0 & I_n & 0 \\ \hline 0 & 0 & I_m  \end{array} \right].
\end{equation} }
Let us define the submatrix of \eqref{Sz}, given by
\begin{equation} \label{prls}
W(X,\M) := \left[
\begin{array}{cc}
X - A^{\mathsf{H}}XA & C^{\mathsf{H}} - A^{\mathsf{H}}XB \\
C- B^{\mathsf{H}}XA & D^{\mathsf{H}}+D -B^{\mathsf{H}}XB
\end{array}
\right],
\end{equation}
which we will also denote as $W(X)$ when the underlying model $\M$ is obvious from the context. Then it follows by simple algebraic manipulation that
\[
\Phi(z) =
\left[ \begin{array}{cc} B^{\mathsf{H}}(z^{-1}\,I_n - A^{\mathsf{H}})^{-1} & I_m  \end{array} \right]
\, W(X,\M) \left[ \begin{array}{c} (z\,I_n -A)^{-1}B \\ I_m \end{array} \right],
\]
and that $\mathcal T(z)$ is positive real if and only if there exists $X\in \Hn$ such that the Linear Matrix Inequality (LMI)
\begin{equation} \label{KYP-LMI}
W(X,\M) \geq 0
\end{equation}
holds. Moreover, $\mathcal T(z)$ is stable if and only if the matrix $X$ in this LMI is also positive definite.
We will therefore make frequent use of the following sets
\begin{subequations}\label{LMIsolnsets}
\begin{align}
&\XWpd :=\left\{ X\in \Hn \left|   W(X,\M) \geq 0,\ X >0 \right.\right\},\label{XpdsolnWpsd} \\[1mm]
&\XWpdpd :=\left\{ X\in \Hn \left|   W(X,\M) > 0,\ X >0 \right.\right\}. \label{XpdsolWpd}
\end{align}
\end{subequations}
An important subset of ${\XWpd}$ are those solutions to (\ref{KYP-LMI}) for which the
rank $r$ of $W(X)$ is minimal ({\ie} for which $r=\rank\Phi(z)$).
If $D^{\mathsf{H}}+D-B^{\mathsf{H}}XB$ is invertible, then
the minimum rank solutions in $\XWpd$
are those for which $\rank W(X) = \rank (D^{\mathsf{H}}+D-B^{\mathsf{H}}XB)  = m$, which in turn is the case
if and only if the Schur complement of $D^{\mathsf{H}}+D-B^{\mathsf{H}}XB$ in $W(X)$ is zero.  This Schur
complement is associated with the discrete-time \emph{algebraic Riccati equation (ARE)}
\begin{equation}
\mathsf{Ricc}(X) := X -A^{\mathsf{H}}XA - (C^{\mathsf{H}}-A^{\mathsf{H}}XB)(D^{\mathsf{H}}+D-B^{\mathsf{H}}XB)^{-1}(C-B^{\mathsf{H}}XA)=0.\label{riccatid}
\end{equation}
Solutions $X$ to (\ref{riccatid}) produce a spectral factorization of $\Phi(z)$, and each solution corresponds to a
\emph{invariant subspace} spanned by the columns of $U:=\matr{cc} I_n & -X^{\mathsf{T}} \rix^{\mathsf{T}} $
that remains invariant under the multiplication with the matrix
\begin{equation}\label{SymMatrix}
S :=\matr{cc} I_n &  B (D^{\mathsf{H}}+D)^{-1} B^{\mathsf{H}} \\
 0 & (A-B (D^{\mathsf{H}}+D)^{-1} C)^{\mathsf{H}} \rix^{-1} \matr{cc} A-B (D^{\mathsf{H}}+D)^{-1} C & 0 \\
C^{\mathsf{H}} (D^{\mathsf{H}}+D)^{-1} C & I_n \rix,
\end{equation}
\ie $U$ satisfies $SU=U A_F$ where the so-called \emph{closed loop matrix} is defined as $A_{F}=A-BF$ with
$F := (D^{\mathsf{H}}+D-B^{\mathsf{H}}XB)^{-1}(C-B^{\mathsf{H}}XA)$. Such a subspace is called a \emph{Lagrangian invariant subspace} and the
matrix $S$ has a \emph{symplectic structure} (see e.g., \cite{Meh91},\cite{FreMX02}).
Each solution $X$ of~\eqref{riccatid} can also be associated with an \emph{extended Lagrangian invariant subspace}
for the pencil $S(z)$, spanned by the columns of
$ \widehat{U}:=\matr{ccc} -X^{\mathsf{T}}
& I_n & -F^{\mathsf{T}}  \rix^{\mathsf{T}}$.
 In particular, $\widehat{U}$ satisfies
\[
\left[ \begin{array}{ccc} 0 & A & B \\
	-I_n & 0 & C^{\mathsf{H}} \\  0 & C & D^{\mathsf{H}}+D  \end{array} \right] \widehat{U}
  =\left[ \begin{array}{ccc} 0 & -I_n & 0\\
	A^{\mathsf{H}} & 0 & 0\\ B^{\mathsf{H}} & 0 & 0 \end{array} \right] \widehat{U} A_{F}.
\]
If $D^{\mathsf{H}}+D-B^{\mathsf{H}}XB$ is singular
then more complicated constructions are necessary, see \cite{Meh91}.

In the continuous-time case, the definition of a passive systems has its origin in network theory, but its formal definition is
associated with the existence of a storage function and a particular dissipation inequality.
The equivalent concept for the discrete-time case again follows from the LMI \eqref{KYP-LMI}.
If we define the vector $z_k$ as the stacked vector of the state $x_k$ above the input $u_k$, and construct the inner product
$z_k^{\mathsf{H}} W(X) z_k$, then we obtain the inequality
\begin{equation} \label{dissipation}
 x_k^{\mathsf{H}}X x_k - x_{k+1}^{\mathsf{H}} Xx_{k+1} + y_k^{\mathsf{H}}u_k + u_k^{\mathsf{H}}y_k =  z_k^{\mathsf{H}}W(X) z_k \ge 0.
\end{equation}
Using the quadratic storage function $\mathcal H(x_i):=\frac12 x_i^{\mathsf{H}}X x_i$ this yields a \emph{dissipation inequality}
\[
 \mathcal H(x_{k})- \mathcal H(x_{0}) \le \sum_{i=0}^{k-1} \Re(y_i^{\mathsf{H}}u_i)
\]
that is similar to the one of the continuous-time formulation.
It follows from the continuous-time literature \cite{Wil72b} and the bilinear transformation between continuous-time and discrete-time systems \cite{BMNV}
that if the system $\M$ of (\ref{statespace}) is minimal, then the LMI \eqref{KYP-LMI}
has a solution $X \geq 0$ if and only if $\M$ is a passive system. Moreover, the solutions of \eqref{KYP-LMI} also satisfy the matrix inequalities
\begin{equation} \label{Xbounds}
 0 < X_- \leq X \leq X_+.
\end{equation}
The matrices $X$ satisfying the matrix inequalities \eqref{Xbounds} also form a convex set, which we call $\mathbb{X}^\pm$. We thus have the following inclusions
$$ \XWpdpd  \subset \XWpd  \subset \mathbb{X}^\pm
$$
which implies that all matrices in the sets $\XWpdpd$ and $\XWpd$ are bounded. Notice also that the $(1,1)$ block in the LMI \eqref{prls},\eqref{LMIsolnsets}
is a discrete-time Lyapunov equation with $X>0$. This implies that $A$ is asymptotically stable if $W(X) > 0$ and is stable if $W(X)\geq 0$, see also \cite{LanT85}.
It is also known that if the system is strictly passive, meaning that $\Phi(e^{\imath \omega})>0$ for the whole unit circle, then $X_- < X_+$.

\begin{remark}\label{rem:rm1}{\rm
The bilinear transformation between continuous-time and discrete-time systems preserves the solution sets $\XWpdpd$ and $\XWpd$
as well as the solutions $X_-$ and $X_+$ of the Riccati equation. It was shown, see \eg \cite{MeV19}, that the set $\mathbb{X}^\pm$ has a nonempty interior if and only if $X_- < X_+$. Since  $\XWpd$ is a subset of $\mathbb{X}^\pm$ it also follows $\XWpd$ has an empty interior when $X_+-X_-$ is singular.
}
\end{remark}

\section{Normalized passive realizations} \label{sec:norm}
A special class of realizations of discrete-time passive systems, are the ones associated to a normalized storage function $\mathcal H(x_k)=\frac12\|x_k\|_2^2$.

\begin{definition} \label{pH}
A {\em normalized passive system} has the state-space form \eqref{statespace_d}
where the system matrices satisfy the matrix inequality
\begin{equation}
 \left[ \begin{array}{cc} I_n & C^{\mathsf{H}} \\ C & D^{\mathsf{H}} +D \end{array} \right] -
 \left[ \begin{array}{cc} A^{\mathsf{H}} \\ B^{\mathsf{H}} \end{array} \right] \left[ \begin{array}{cc} A & B \end{array} \right] \ge 0.
\end{equation}
\end{definition}

We now show that every passive system has an equivalent normalized passive realization.
Consider a minimal state-space model  $\M:=\{A,B,C,D\}$ of a passive linear time-invariant system and let  $X\in \XWpd$ be a solution of the LMI \eqref{KYP-LMI}.
We then use a (Cholesky like)  factorization $X= T^{\mathsf{H}}T$ which implies $\det T \neq 0$ and define a new realization
\[
\M_T:=\{A_T,B_T,C_T,D_T\} := \{TAT^{-1}, TB, CT^{-1}, D \}
\]
so that
\begin{eqnarray}
\left[ \begin{array}{cccc} T^{-\mathsf{H}} & 0\\ 0 & I_m
\end{array}
\right]
\left[ \begin{array}{cccc} X-A^{\mathsf{H}}XA & C^{\mathsf{H}}-A^{\mathsf{H}}XB \\ C-B^{\mathsf{H}}XA & D^{\mathsf{H}}+D-B^{\mathsf{H}}XB
\end{array}
\right]
\left[ \begin{array}{cccc} T^{-1} & 0\\ 0 & I_m
\end{array}
\right] \\
=
\left[\begin{array}{cccc} I_n & C_T^{\mathsf{H}} \\ C_T & D_T^{\mathsf{H}}+D_T\end{array}\right]-
\left[\begin{array}{cc} A_T^{\mathsf{H}} \\ B_T^{\mathsf{H}} \end{array}\right]\left[\begin{array}{cc} A_T & B_T \end{array}\right]
\geq 0,
\end{eqnarray}
which expresses that the transformed realization $\M_T$ is now normalized.
Notice that the factor $T$ is unique up to a unitary factor $U$ since $T^{\mathsf{H}}U^{\mathsf{H}}UT=T^{\mathsf{H}}T$.
This unitary factor does not affect the normalization constraint, but we can choose it to put $A_T$ in a special coordinate system.
Notice that the inequality $I_n-A_T^{\mathsf{H}}A_T\ge 0$ implies that $A_T$ is contractive and has a singular value decomposition
$A_T=U\Sigma V^{\mathsf{H}}$ where $0 \le \Sigma \le I_n$. The additional unitary similarity transformation $\{U^{\mathsf{H}}A_TU,U^{\mathsf{H}}B_T,C_TU,D_T\}$
will then yield a new normalized coordinate system $\{A_{\hat T},B_{\hat T},C_{\hat T},D_{\hat T}\}$ where, in addition, $A_{\hat T}=\Sigma (V^{\mathsf{H}}U)$,
which is a polar decomposition with a positive semidefinite Hermitian factor
$\Sigma$ that is diagonal and satisfies  $0 \le \Sigma \le I_n$ \cite{Higham86}.

Even after the normalization,  there is typically still a lot of freedom in the representation of the system, since we could have used any matrix $X$ from the set $\XWpd$ to normalize our realization.
In the remainder of this paper, we will focus on normalized passive realizations. The freedom remaining is thus the choice of the matrix $X$ from $\XWpd$, which, as we will see, can be used to make the representation more robust, i.e., less sensitive to perturbations. The remainder of this paper will deal with the question of how to make use of this freedom in the state space transformation to determine a 'good' or `nearly optimal' normalized realization.

\section{The passivity radius}\label{sec:passrad}
Our goal is to achieve `good' or `nearly optimal' normalized realizations of a passive system. A natural measure for this is a large
\emph{passivity radius} $\rho_{\M}$, which is the smallest perturbation (in an appropriate norm) to the coefficients of a model $\M $ that causes the perturbed system to loose this property.

Once we have determined a solution $X\in \XWpd$ to the LMI~(\ref{KYP-LMI}), we can determine the normalized representations as discussed in Section~\ref{sec:norm}.
For each such representation we can determine the passivity radius and then choose the solution $X\in \XWpd$ which is most
robust under perturbations ${\Delta_\M}$ of the model parameters $\M:=\{A,B,C,D\}$. This is a suitable approach for perturbation analysis,
since as soon as we fix $X\in \XWpdpd$, we will see that we can solve for the smallest perturbation $\Delta_\M$ to our model $\M $ that makes
$\det W(X,\M+\Delta_\M)=0$.  To measure the size of the perturbation $\Delta_\M$ of a state space model $\M $ we will use the Frobenius norm
or the 2-norm of the matrix $\Delta_\S$ defined as
\begin{equation} \label{DeltaS}
 \Delta_\S:= \left[\begin{array}{ccc}
\Delta_A & \Delta_B \\
\Delta_C & \Delta_D
\end{array}\right]
\end{equation}
and we use also the notion of \emph{$X$-passivity radius}, which was introduced in \cite{BMV18}, and gives a bound for the usual passivity radius.
\begin{definition}
For $X\in \XWpdpd$ the  \emph{$X$-passivity radius} is defined as
\[
	\rho_\M(X):= \inf_{\Delta_\S\in \mathbb C^{n+m,n+m}}\left\{ \| \Delta_\S \| \; | \; \det W(X,\M+\Delta_\M) = 0\right\}.
\]
\end{definition}
Note that in order to compute $\rho_\M(X)$ for the model $\M $, we must have a point $X\in \XWpdpd$, since $W(X,\M)$  must be positive definite to start with and also $X$ should be positive definite to obtain a state-space transformation from it. The following relation between the $X$-passivity radius and the usual passivity radius was already presented in \cite{BMV18}.
\begin{lemma} \label{lem:passrad}
	The passivity radius for a given model $\M$ satisfies
	\begin{equation}
\nonumber
	\rho_{\M}:= \sup_{X\in \XWpdpd}\inf_{\Delta_\S\in \mathbb C^{n+m,n+m}}\{\| \Delta_\S \| | \det W(X,\M+\Delta_\M)=0\}= \sup_{X\in \XWpdpd} \rho_{\M}(X).\label{passive}
	\end{equation}
\end{lemma}

We now provide an exact formula for the $X$-passivity radius based on a one parameter optimization problem. For this, we point out that the condition $W(X,\M+\Delta_\M)>0$ is
equivalent to the condition
\begin{equation}
\widehat W(X,\M+\Delta_\M):= \left[\begin{array}{ccc} X^{-1} &  A + \Delta_A & B+\Delta_B \\  A^{\mathsf{H}}+\Delta_A^{\mathsf{H}} & X & C^{\mathsf{H}}+\Delta_C^{\mathsf{H}}  \\ B^{\mathsf{H}}+\Delta_B^{\mathsf{H}} &
C+\Delta_C & D^{\mathsf{H}}+\Delta_D^{\mathsf{H}} + D+\Delta_D  \end{array}\right] > 0,
\end{equation}
which is now an LMI in the unknown parameters of $\Delta_\M$ (for a fixed $X$).
Setting
\begin{equation} \label{defWhatX}
\widehat W:=\widehat W(X,\M)=\left[\begin{array}{ccc} X^{-1} &  A & B \\  A^{\mathsf{H}} & X & C^{\mathsf{H}} \\ B^{\mathsf{H}} &
C & D^{\mathsf{H}}+ D  \end{array}\right], \quad E:= \left[\begin{array}{c|c} E_1 & E_2  \end{array}\right] \left[\begin{array}{cc|cc} I_n & 0 & 0 & 0 \\ 0 & 0 & I_n & 0 \\ 0 & I_m & 0 & I_m \end{array}\right],
\end{equation}
and using the matrix $\Delta_\S$ in \eqref{DeltaS}, this inequality can be written as the structured LMI
\begin{equation} \label{WDelta}
\widehat W+ E \left[\begin{array}{cc} 0 & \Delta_\S \\ \Delta_\S^\mathsf{H} & 0
\end{array}\right]E^\mathsf{T} > 0
\end{equation}
as long as the system is still passive. In order to violate this condition, we need to find the smallest $\Delta_\S$ such that the determinant of \eqref{WDelta} becomes 0.
Since $\widehat W$ is positive definite, we can then construct its Cholesky factorization $\widehat W := R^\mathsf{H}R$. The matrix in \eqref{WDelta} will become singular when the matrix
\begin{equation} \label{WDelta2}
I_{2n+m}+ R^{-\mathsf{H}} E \left[\begin{array}{cc} 0 & \Delta_\S \\ \Delta_\S^\mathsf{H} & 0
          \end{array}\right]E^\mathsf{T} R^{-1}
\end{equation}
becomes singular. The following theorem, is analogous to results obtained for continuous-time systems \cite{BMV18,MeV19,OvV05}, and
we therefore omit the proof. It gives for this kind of problem the minimum norm perturbation $\Delta_\S$ both in Frobenius norm and in 2-norm.
\begin{theorem} \label{thm:mingamma}
Consider the matrices $\hat X, \widehat W=R^\mathsf{H}R$ in (\ref{defWhatX}) and the pointwise positive semidefinite matrix function
\begin{equation}\label{defmgamma}
M(\gamma):= \left[ \begin{array}{cc} \gamma F_1^{\mathsf{H}}  \\ \gamma^{-1}F_2^{\mathsf{H}} \end{array}\right]
\left[ \begin{array}{cc} \gamma F_1 & \gamma^{-1} F_2 \end{array}\right],  \;\;
F_1:=R^{-\mathsf{H}}E_1, \;\; F_2:=R^{-\mathsf{H}} E_2
\end{equation}
in the real parameter $\gamma \in(0,\infty)$. Then the largest eigenvalue $\lambda_{\max}(M(\gamma))$ is a \emph{unimodal function} of $\gamma$ ({i.e.} it is first monotonically decreasing and then monotonically increasing with growing $\gamma$). At the minimizing value $\underline \gamma$,  $M(\underline{\gamma})$ has an eigenvector $z$, {i.e.}
\[
 M(\underline{\gamma}) z = \underline\lambda_{\max} z, \quad z:=\left[ \begin{array}{cc} u \\ v \end{array}\right],
 \]
where
$  \|u\|_2^2=\|v\|_2^2=1$.
The minimum norm perturbation $\Delta_\S$ is of rank $1$ and is given by $\Delta_\S=uv^{\mathsf{H}}/\underline{\lambda}_{\max}$. It has norm $1/\underline{\lambda}_{\max}$ both in 2-norm and in Frobenius norm.
\end{theorem}
A simple bound for $\underline{\lambda}_{\max}$ can also be obtained, as pointed out in \cite{BMV18} for the continuous-time case. The proof is essentially the same and is therefore omitted.
\begin{corollary}\label{cor:lev} Consider the matrices $\widehat W$, $F_1$, $F_2$ and $M(\gamma)$ in Theorem \ref{thm:mingamma},
and define $\alpha:=\|F_1\|_2$ and $\beta:=\|F_2\|_2$.
Then the norm of $M(\gamma)$ is also the norm of $\gamma^2 F_1F_1^{\mathsf{H}} + \gamma^{-2}F_2F_2^{\mathsf{H}}$, and
\[ \underline{\lambda}_{\max} = \|M(\underline{\gamma})\|_2 = \min_{\gamma>0} \|M(\gamma)\|_2 = \min_{\gamma>0}
\|\gamma^2 F_1F_1^{\mathsf{H}} + \gamma^{-2}F_2F_2^{\mathsf{H}}\|_2\le 2\|F_1\|_2\|F_2\|_2=2\alpha\beta.
\]
This upper bound is reached if and only if
the matrices $F_1F_1^{\mathsf{H}}$ and $F_2F_2^{\mathsf{H}}$ have a common eigenvector associted with the maximal eigenvalue.
\end{corollary}

The following theorem is a variant of a result proven in \cite{BMV18}, and constructs a rank one perturbation which makes the matrix
$W_{\M+\Delta_\M}$ singular and therefore gives an upper bound for $\rho_M(X)$.

\begin{theorem}\label{thm:Xpassivity}
Let $\M=\{A,B,C,D\}$ be a given minimal passive discrete-time model and assume that we are given a matrix $X\in \XWpdpd$, then the $X$-passivity radius $\rho_\M(X)$
is bounded by
$$ 1/(2\alpha\beta) \le \rho_\M(X) \le  1/[(1+|\hat v^{\mathsf{H}}\hat u|)(\alpha\beta)]\le 1/(\alpha\beta),$$
where $\hat u, u$ and $\hat v, v$ are normalized dominant singular vector pairs of $F_1$ and $F_2$, respectively~:
$$F_1 u=\alpha \hat u, \;\; F_1^{\mathsf{H}} \hat u=\alpha u, \;\; F_2 v=\beta \hat v, \;\; F_2^{\mathsf{H}} \hat v=\beta v.$$
Moreover, if $\hat u$ and $\hat v$ are linear dependent, then $\rho_\M(X)=1/(2\alpha\beta)$.
\end{theorem}
\begin{proof} The proof is analogous to the continuous-time case, see \cite{MeV19}.
\end{proof}

Finally, we point out here that in order to maximize the passivity radius of a
system model $\M$, one should maximize the smallest eigenvalue of the scaled
matrix $\widehat W(X,\M)$. Let $D_s=\diag(I_n,I_n,I_m/\sqrt{2})$ and let us scale the
inequality \eqref{WDelta} with the matrix $D_s$~ given by
\begin{equation} \label{Wscale} D_s \widehat W(X,\M) D_s + D_s E \left[\begin{array}{cc} 0 & \Delta_\S \\ \Delta_\S^\mathsf{H} & 0
\end{array}\right]E^\mathsf{T} D_s
\end{equation}
where now $D_sE$ is an isometry. It then follows that in order to have a perturbation
$\Delta_\S$ of norm $\rho_\M(X)$ that makes \eqref{Wscale} singular, we must have
\begin{equation} \label{rhobound}
 \lambda_{\min}(D_s\widehat W(X,\M) D_s)\le \rho_\M(X).
\end{equation}
This bound expresses that if we want to maximize $\rho_\M(X)$ over all $X\in \XWpd$, we should try to maximize $\lambda_{\min}(D_s\widehat W(X,\M)D_s).$
The following result shows that normalized passive realizations can be expected to have a larger minimal eigenvalue in the matrix $D_s\widehat W(I,\M_T)D_s$ than the corresponding minimal eigenvalue of the non-normalized matrix $D_s\widehat W(X,\M)D_s$.
\begin{lemma}
Let $X\in \XWpdpd$ then the trace of the matrix
\[  \min_{\det T\neq 0} \trace[\diag(T,T^{-\mathsf{H}},I_m)(D_s\widehat W(X,\M)D_s) \diag(T^\mathsf{H},T^{-1},I_m)]= \trace (D_s\widehat W(I,\M_T)D_s)
\]
is minimized by the matrices $T$ such that $X=T^\mathsf{H}T$, while the determinant remains invariant
\[ \det[\diag(T,T^{-\mathsf{H}},I_m)(D_s\widehat W(X,\M)D_s) \diag(T^\mathsf{H},T^{-1},I_m)]= \det (D_s\widehat W(I,\M_T)D_s).
\]
\end{lemma}
\begin{proof} Note that transformation applied to $D_s\widehat W(X,\M)D_s$ is a congruence transformation
which preserves the nonnegativity of its eigenvalues and that the trace of the resulting matrix is
 $\trace Z + \trace Z^{-1} + \frac12\trace(D^\mathsf{H}+D)$, where $Z:=TX^{-1}T^\mathsf{H}$. It is well known that this is minimized when $Z=I$.
The fact that the congruence transformation preserves the determinant identity is obvious.
\end{proof}
This lemma suggests that the smallest eigenvalue should increase as the product of all the eigenvalues remains constant and their sum is being minimized,
but this is of course not guaranteed in general.

\section{Maximizing the passivity radius} \label{sec:maxpass}

In this section we discuss another LMI in the matrices $X>0$ with the same domain as $W(X,\M)\ge 0$, given by
\[
\widetilde W(X,\M) :=\left[\begin{array}{ccc} X &  XA & XB \\  A^{\mathsf{H}}X & X & C^{\mathsf{H}} \\ B^{\mathsf{H}}X &
C & D^{\mathsf{H}}+ D  \end{array}\right] \ge 0.\]
It is clear that $\widetilde W(X,\M)$ is congruent to $\diag(X,W(X,M))$ and since $X>0$, it has the same solution set $\XWpd$ as $W(X,M)\ge 0$.
The LMI for the normalized passive realization $\M_T=\{TAT^{-1},TB,CT^{-1},D\}$  corresponding to $X=T^\mathsf{H}T$, can be obtained via a congruence transformation as well
{\small
\[
\widetilde W(I,\M_T) := \left[\begin{array}{ccc} I_n & A_T & B_T \\  A_T^{\mathsf{H}} & I_n & C_T^{\mathsf{H}} \\ B_T^{\mathsf{H}} &
C_T & D_T^{\mathsf{H}}\!+\!D_T  \end{array}\right] = \left[\begin{array}{ccc}  T^{-{\mathsf{H}}} & 0 & 0 \\ 0 & T^{-{\mathsf{H}}} & 0 \\ 0 & 0 & I_m \end{array}\right] \widetilde W(X,\M)  \left[\begin{array}{ccc}  T^{-1} & 0 & 0 \\ 0 & T^{-1} & 0 \\ 0 & 0 & I_m \end{array}\right]\ge 0.
\]
}
Let us now consider the following constrained LMI
\begin{equation} \label{xi}
 \widetilde W(X,\M) \ge \xi \diag(X,X,2I_m).
\end{equation}
Then the following Theorem gives a bound on how large we can choose $\xi$ in this LMI.
\begin{theorem}  \label{thm:maxoverX}
Let $\M:=\{A,B,C,D\}$ be a minimal realization of a discrete-time passive system, and let $X$ be any matrix in $\XWpd$. Then there is a unique $\xi^*(X)$ which is maximal for the matrix inequality \eqref{xi} to hold, and which is strictly smaller that 1. Moreover, $\xi^*(X) =\lambda_{\min} (D_s\widetilde W(I,\M_T)D_s).$
\end{theorem}
\begin{proof}
It follows from \eqref{Xbounds} that every $X\in\XWpd$ is positive definite. Therefore it can be factorized as $X=T^{\mathsf{H}}T$ with $\det T\neq 0$, and we can consider the normalized system $\M_T=\{TAT^{-1},TB,CT^{-1},D\}$. It is easy to see that the condition \eqref{xi} is equivalent to the corresponding LMI condition for the transformed system $\M_T$, which is given by
\[
\widetilde W(I,\M_T)  \ge \xi \diag(I_n,I_n,2I_m).
\]
The largest value $\xi^*(X)$ of $\xi$ for which this holds is clearly equal to
\begin{equation} \xi^*(X)  =  \max_\xi\left[ \;\xi\; | \;  D_s\widetilde W(I,\M_T)D_s \ge \xi  I_{2n+m}\right] = \lambda_{\min} (D_s\widetilde W(I,\M_T)D_s).
\end{equation}
Since $D_s\widetilde W(I,\M_T)D_s-\xi^*I_{2n+m}$ is positive semi-definite, its diagonal must be non-negative, and thefore $\xi^*$ can not be larger than 1.
Moreover, $\xi^*=1$ would imply then that $A_T$, $B_T$ and $C_T$ would be zero.
\end{proof}

\begin{remark}\label{rem:eq}{\rm
Note that $\widetilde W(I,\M_T)=\widehat W(I,\M_T)$. From \eqref{rhobound} one then obtains the inequality
\[ \lambda_{\min} (D_s\widetilde W(I,\M_T)D_s) \le \rho_{\M_T}
\]
which shows the relevance of $\widetilde W(I,\M_T)$ in the maximization of the passivity radius.
}
\end{remark}

The use of the characterization $\xi^*(X) := \lambda_{\min} D_s\widetilde W(I,\M_T)D_s$ in terms of the LMI \eqref{xi} is crucial for the rest of this section.
We also point out that Theorem~\ref{thm:maxoverX} applies to all points of $\XWpd$, and therefore also of  $\XWpdpd$. But we can distinguish between both.
\begin{corollary} \label{distinction}
The maximal value $\xi^*(X)$ of a matrix $X\in \XWpd$ for a given model $\M$ equals 0 if $X$ is a boundary point of $\XWpd$ and is strictly positive if and only if $X$ is in $\XWpdpd$.
\end{corollary}
\begin{proof}
If $X$ is a boundary point of $\XWpd$ then $\det W(X,\M)=0$ and also $\det \widetilde W(X,\M)=0$ and for those $X$, we thus have $\xi^*(X)=0$.
If $X$ belongs to $\XWpdpd$, then $\widetilde W(X,\M)>0$ and $\diag(X,X,2I_m)>0$.
Therefore there exists an $\xi>0$ such that $\widetilde W(X,\M)>\xi\diag(X,X,2I_m)$, and hence $\xi^*(X)>0$. Conversely, if
$\xi^*(X)>0$ then $\widetilde W(X,\M)>0$ and $W(X,\M)>0$ which implies that $X\in \XWpdpd$.
\end{proof}

In order to maximize $\xi^*(X)$, we consider for a given $X$ in  $\XWpd$ the matrix
\[
\widetilde W(X,\M_\xi):= \left[\begin{array}{ccc} X  &  XA_\xi & XB_\xi \\
A_\xi^{\mathsf{H}}X & X & C_\xi^{\mathsf{H}} \\   B_\xi^{\mathsf{H}}X & C_\xi & D_\xi^{\mathsf{H}}+D_\xi \end{array}\right]
\]
corresponding to the modified model
$\M_\xi :=\{A_\xi,B_\xi,C_\xi,D_\xi\}:=\{\frac{A}{(1-\xi)},\frac{B}{(1-\xi)},\frac{C}{(1-\xi)},\frac{D-\xi I_m}{(1-\xi)}\}$.
It turns out that this matrix satisfies the identity
\begin{equation} \label{shifted}
(1-\xi)\widetilde W(X,\M_\xi) = \widetilde W(X,\M) - \xi \left[\begin{array}{ccc} X  & 0 & 0 \\ 0 & X & 0 \\ 0 & 0 & 2I_m \end{array}\right]
\end{equation}
which is crucial for the following Lemma.
\begin{lemma} \label{inclusion}
For every $X> 0$ in $\XWpdpd$ and any $0\le \xi_- < \xi_+  \le \xi^*(X)$, the passivity LMIs for the systems $\M_{\xi_-}$ and $\M_{\xi_+}$ are satisfied. Moreover, the solution set of
$\widetilde W(X,\M_{\xi_+})\ge 0$ is included in the solution set of  $\widetilde W(X,\M_{\xi_-})> 0$.
\end{lemma}
\begin{proof}
The LMIs for two different values of $\xi$ are related as
\[
 (1-\xi_2)\widetilde W(X,\M_{\xi_2}) = (1-\xi_1)\widetilde W(X,\M_{\xi_1}) - (\xi_2-\xi_1) \diag(X,X,2I_m).
\]
Since $X\in\XWpdpd$, we have that $\xi^*(X)>0$ and $\diag(X,X,2I_m)>0$. For that $X$, it then follows that
\begin{equation} \label{ineqs}
\widetilde W(X,\M) \ge (1-\xi_-)\widetilde W(X,\M_{\xi_-}) > (1-\xi_+)\widetilde W(X,\M_{\xi_+}) \ge (1-\xi^*(X))\widetilde W(X,\M_{\xi^*(x)}) \ge 0.
\end{equation}
The systems  $\M_{\xi_-}$ and $\M_{\xi_+}$ are thus passive, since their associated LMIs have a nonempty solution set. Now consider {\em any} $X$ for which $\widetilde W(X,\M_{\xi_+})\ge 0$. Since $\xi_+$ is strictly positive, so is $\xi^*(X)$ and hence $X\in \XWpdpd$. It then follows from \eqref{ineqs} that $\widetilde W(X,\xi_-)>0$. Hence,
the solution set of $\widetilde W(X,\M_{\xi_+}) \ge 0$ is included in the solution set of $\widetilde W(X,\M_{\xi_-}) > 0$.
\end{proof}

Lemma~\ref{inclusion} implies that for a given $X\in \XWpdpd$, the solution sets of  $\widetilde W(X,\M_\xi) \ge 0$ are shrinking with increasing $\xi$.
But we still need to find the matrix $X\in \XWpd$ that maximizes $\xi^*(X)$. We can answer this question by relating this to the passivity of the transfer function of the modified system $\M_\xi$,
\[
  \T_\xi(z):=C_\xi(zI_n -A_\xi)^{-1}B_\xi+D_\xi,
\]
which is minimal since $\M$ was assumed to be minimal.
It follows from the discussion of Section \ref{sec:prelim} that this transfer function corresponds to a {\em strictly} passive system if and only if the conditions
(i) the transfer function  $\T_\xi(z)$ is asymptotically stable, and
(ii) the matrix function $\Phi_\xi(z):=\T^\mathsf{H}_\xi(z^{-1})+\T_\xi(z)$ is strictly positive on the unit circle $e^{\imath \omega}, \omega\in[-\pi,\pi]$,
    are satisfied.
It has been shown in Section \ref{sec:prelim} that the zeros of $\Phi_\xi(z)$ are the eigenvalues of the symplectic matrix
\begin{equation}
S_\xi:=\matr{cc} I_n &  B_\xi (D_\xi^{\mathsf{H}}+D_\xi)^{-1} B_\xi^{\mathsf{H}} \\
 0 & (A_\xi-B_\xi (D_\xi^{\mathsf{H}}+D_\xi)^{-1} C_\xi)^{\mathsf{H}} \rix^{-1} \matr{cc} A_\xi-B_\xi (D_\xi^{\mathsf{H}}+D_\xi)^{-1} C_\xi & 0 \\
C_\xi^{\mathsf{H}} (D_\xi^{\mathsf{H}}+D_\xi)^{-1} C_\xi & I_n \rix,
\end{equation}
which are also the finite eigenvalues of the pencil
\begin{equation*}
z  \left[ \begin{array}{ccc}  0 & -I_n & 0 \\ A_\xi^\mathsf{H} & 0 & 0 \\
B_\xi^{\mathsf{H}} & 0 & 0 \end{array}
\right] + \left[ \begin{array}{ccc} 0 & A_\xi & B_\xi \\ -I_n & 0 & C^{\mathsf{H}}_\xi \\
0 & C_\xi & D^{\mathsf{H}}_\xi+D_\xi \end{array}\right]
\end{equation*}
or equivalently, those of the pencil
\begin{equation} \label{xipencil}
z  \left[ \begin{array}{ccc}  0 & (\xi-1) I_n & 0 \\ A^\mathsf{H} & 0 & 0 \\
B^{\mathsf{H}} & 0 & 0 \end{array}
\right] + \left[ \begin{array}{ccc} 0 & A & B \\ (\xi-1)I_n & 0 & C^{\mathsf{H}} \\
0 & C & D^{\mathsf{H}}+D - 2\xi I_m \end{array}\right]
\end{equation}
and that the realization of $\M_\xi$ is minimal.
The algebraic conditions corresponding to strict passivity of $\T_\xi(z)$ are therefore
\begin{enumerate}
	\item[A1.] $A_\xi $ has all its eigenvalues inside the unit disc (stability),
	\item[A2.] the pencil \eqref{xipencil} has no eigenvalues on the unit circle  (positive realness).
\end{enumerate}
These conditions are phrased in terms of eigenvalues of certain matrices that depend on the parameter $\xi$. Since eigenvalues are continuous functions of the matrix elements, one can consider limiting cases for the above conditions. As explained in Section \ref{sec:prelim} the passive transfer functions are limiting cases of strictly passive ones. Those limiting cases correspond to the value of $\xi$ where one of the conditions A1. or A2. does not hold anymore.
\begin{theorem} \label{thm:Xi} Let $\M$ be a strictly passive and minimal system. Then there is a bounded supremum $\Xi:=\sup_\xi \{\xi \; | \; \T_\xi(z) \mathrm{\; is \; strictly \; passive}\}$ for which the following properties hold
\begin{enumerate}
\item  $\T_\Xi(z)$ is passive,
\item the solution set of $\widetilde W(X,\M_\Xi)\ge 0$ is not empty,
\item the solution set of $\widetilde W(X,\M_\Xi)> 0$ is empty,
\item for any $\xi < \Xi$ the solution set of $\widetilde W(X,\M_\xi)> 0$ is non-empty,
\item  $\Xi:=\sup_X \xi^*(X)$ for all $X\in \XWpd$.
\end{enumerate}
\end{theorem}
\begin{proof}
	The existence of a bounded supremum follows from the fact that $\T_\xi(z)$ is strictly passive
	only if $\xi$ is smaller than 1 (see Theorem \ref{thm:maxoverX}). Property 1. holds because
	$\T_\Xi(z)$ is the limit of $\T_\xi(z)$ for $\xi \rightarrow \Xi$. Property 2. is a direct consequence of the previous property. Property 3. follows by contradiction~: if $\widetilde W(X,\M_\Xi)> 0$ would not be empty, then $\xi^*(X)$ for $X$ in the domain of $\widetilde W(X,\M_\Xi)> 0$, would be larger that $\Xi$. Property 4. follows from Lemma \ref{inclusion} where we use any $X$ in the domain of $\widetilde W(X,\M_\Xi)\ge  0$ and choose $\xi_+=(\Xi+\xi)/2$ and $\xi_-=\xi$ to show that $X$ also lies in the domain of $\widetilde W(X,\M_\xi)> 0$. Property 5. follows from
	$\xi^*(X)=\max \{ \xi \; | \; \widetilde W(X,\M_\xi)\ge 0 \}$, which expresses that $\T_\xi(z)$ is passive.
\end{proof}

The following theorem discusses the optimal passivity radius over all realizations of $\T(z)$.
\begin{theorem} \label{thm:optimal}
	Let $\M:=\{A,B,C,D\}$ be a minimal realization of a strictly passive transfer function $\T(z):= C(zI-A)^{-1}B+D$. 	Then
\[
\Xi:=\sup_\xi \{\xi \; | \; \T_\xi(z) \mathrm{\; is \; strictly \; passive}\}
\]
 is a lower bound for the largest possible passivity radius within the set of all realizations of $\T(z)$.
 Moreover, normalized realizations $\M_T:=\{T^{-1}AT,BT,T^{-1}C,D\}$, where $X:=T^\mathsf{H}T$ corresponds to a solution $X$ of $\widetilde W(X,\M_\Xi)\ge 0$,
 have a passivity radius  $\rho_{\M_T}$ larger than or equal to $\Xi$.
\end{theorem}
\begin{proof}
Consider realizations  $\M_T:=\{T^{-1}AT,BT,T^{-1}C,D\}$ with $X:=T^\mathsf{H}T$ and $X \in \widetilde W(X,\M)\ge 0$.
It was shown in Theorem~\ref{thm:maxoverX} that for the corresponding realization $\M_T$, we have that $\xi^*(X)=\lambda_{\min}(D_s\widetilde W(I,\M_T)D_s)$.
Theorem~\ref{thm:Xi} then shows that for a solution $X$ of $\widetilde W(X,\M_\Xi)\ge 0$ corresponding to the supremum of all $\xi^*(X)$, we have
$\Xi=\lambda_{\min}(D_s\widetilde W(I,\M_T)D_s)$. The lower bound $\Xi\le \rho_{\M_T}$ then follows from Remark \ref{rem:eq} and Lemma \ref{lem:passrad}.
\end{proof}

\begin{figure}[ht]
	\centering
	\includegraphics[width=10cm]{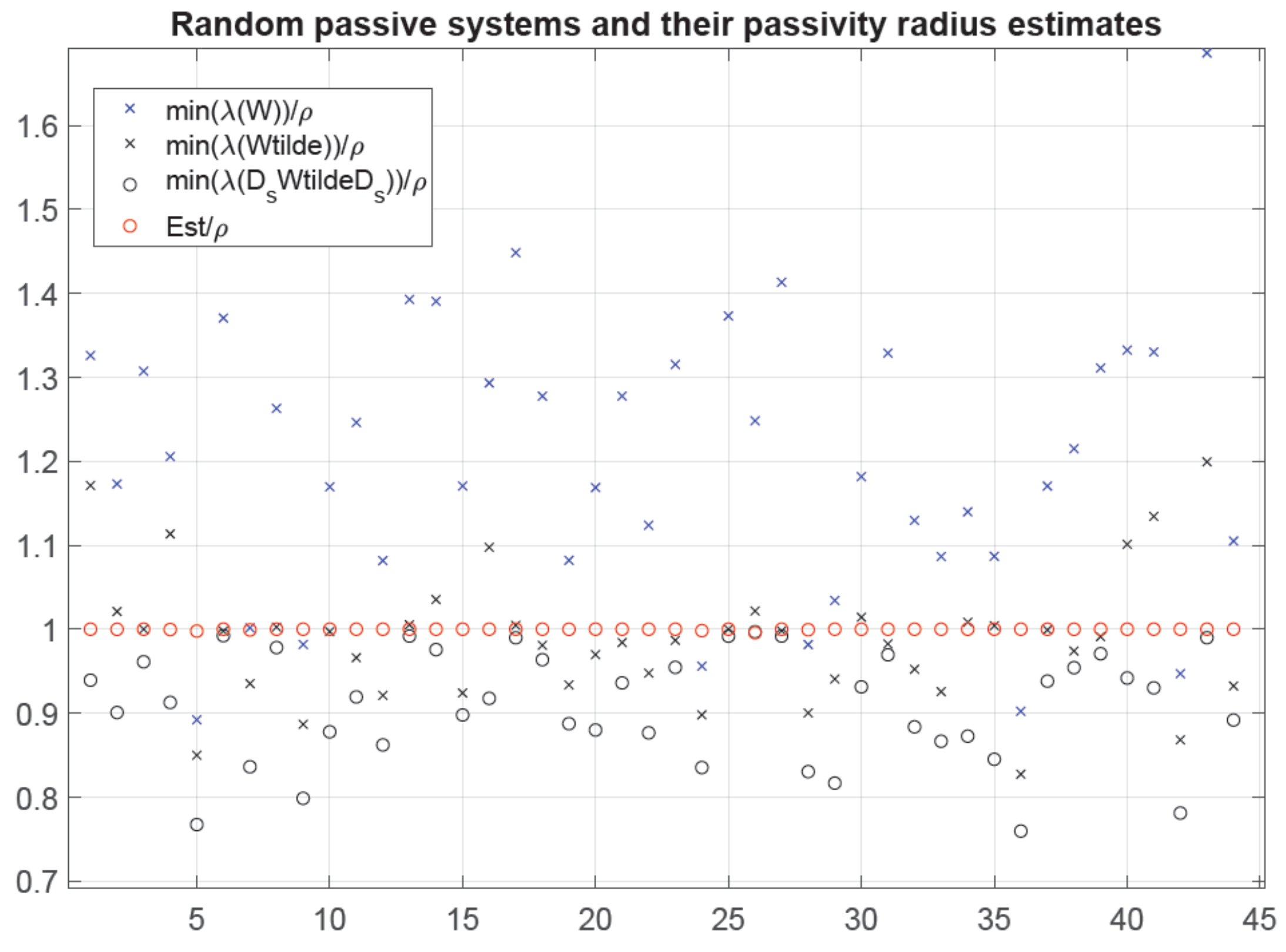}
	\caption{Relative accuracies of four estimates of the passivity radius of a random system~: $\lambda_{\min}W(I,\M_T)$, $\lambda_{\min}\widetilde W(I,\M_T)$, $\lambda_{\min}D_s\widetilde W(I,\M_T)D_s$,
	and $Est=\| \left[ \gamma^{gm} N_1\;|\; N_2/\gamma^{gm} \right] \|^2_2$}
	\label{fig:random}
\end{figure}

In Figure \ref{fig:random}, we generated random normalized passive systems and computed the following quantities (using $\gamma^{gm}:=\sqrt{\beta/\alpha}$ as defined in Appendix B):
\begin{enumerate}
	\item The passivity radius $\rho_{\M_T}$, computed to 4 digits of accuracy,
	\item $\lambda_{\min}W(I,\M_T)$.
	\item $\lambda_{\min}\widetilde W(I,\M_T)$,
	\item $\lambda_{\min}D_s\widetilde W(I,\M_T)D_s$ which is a lower bound for $\rho_{\M_T}$,
	\item $Est:=\| \left[ \gamma^{gm} N_1\;|\; N_2/\gamma^{gm} \right] \|^2_2$ which is also a lower bound for $\rho_{\M_T}$.
\end{enumerate}

In Figure~\ref{fig:random} we depict the quantities (2.-5.) divided by $\rho_{\M_T}$ to indicate their relative
bounds. It can be seen that the eigenvalues are in the interval
\[
 \frac12 \rho_{\M_T} \le \lambda_{\min}W(I,\M_T), \lambda_{\min}\widetilde W(I,\M_T) \le 2\rho_{\M_T}$$
and that $$ \frac12 \rho_{\M_T} \le \lambda_{\min}D_s\widetilde W(I,\M_T)D_s \le \rho_{\M_T},
\quad  \| \left[ \gamma^{gm} N_1\;|\; N_2/\gamma^{gm} \right] \|^2_2 \approx \rho_{\M_T}.
\]
Figure~\ref{fig:random} indicates that $1/g(\gamma^{gm})\le \rho_{\M}(X)$ is a very good estimate of the passivity radius (within $1\%$ of the correct value) and that the bound $\lambda_{\min}(D_s\widetilde WD_s)\le \rho_{\M}(X)$ holds.

\section{A scalar example}

In this section we analyze a simple first order discrete-time scalar system. Its transfer function $T(z)=d+\frac{cb}{z-a}$ is asymptotically stable if $a^2<1$. Then
\[
W(x) =
\left[\begin{array}{cc}
x- a^2x & c -abx\\
c - abx & 2d - b^2x
\end{array}\right]
\]
and the roots $x_-, x_+$ of the quadratic polynomial $\det W(x) = (1-a^2)x(2d-b^2x)-(c-abx)^2$
happen to be the extremal solutions of the associated Riccati equations. The set $\XWpd$ where $W(x)\ge 0$ is thus just the interval $[x_-,x_+]$, provided these two roots are real. This polynomial can be rewritten as
\[
 \det W(x)=-b^2x^2+2\beta x - c^2, \quad \mathrm{where} \quad \beta :=(1-a^2)d+abc
\]
and it has two real roots iff $\beta^2\ge (bc)^2$ or $|\beta/(bc)|= |\frac{(1-a^2)d}{bc} + a|\ge 1$.

The normalized passive realizations are those where we normalize $x$ to 1 by the transformation that scales $\{a,b,c,d\}$ to $\{a,b.t,c/t,d\}$, where $x=t^2\in [x_-,x_+]$.
In Figure~\ref{fig:scalar} we show a plot of the passivity radius of the realizations $\M_t:=\{a,b.t,c/t,d\}$ as a function of $t$, and also the following quantities:
\begin{itemize}
	\item The true passivity radius $\rho_{\M_t}:=1/\lambda_{\max}M(\gamma^*)$ defined in Section \ref{sec:passrad},
	\item $\lambda_{\min} W(I,\M_t)$ which is given in Section \ref{sec:prelim},
	\item $\lambda_{\min}D_s\widetilde W(I,\M_t)D_s$ which is a lower bound for $\rho_{\M_t}$,
	\item the values of $b_t:=b.t$ and $c_t:=c/t$.
\end{itemize}

\begin{figure}[ht]
	\centering
	\includegraphics[width=9cm]{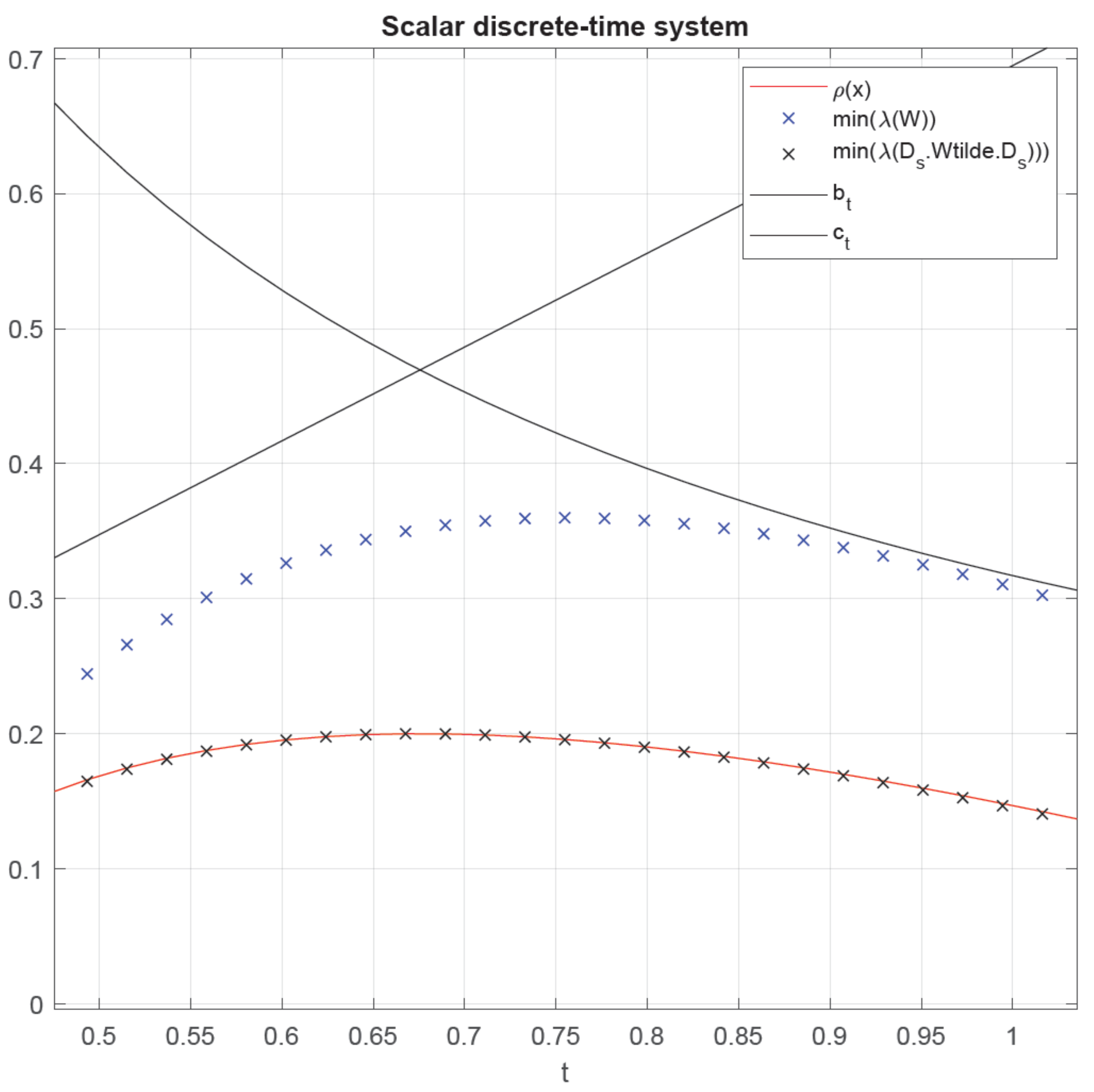}
	\caption{Estimates $\lambda_{\min} W(I,\M_t)$ and $\lambda_{\min}D_s\widetilde W(I,\M_t)D_s$ for the passivity radius $\rho_{\M_t}$ of a scalar normalized system $\M_t:=\{a,b.t,c/t,d\}$ as function of $t$.}
	\label{fig:scalar}
\end{figure}
It is interesting to see that the lower bound $\lambda_{\min}D_s\widetilde W(I,\M_t)D_s$ is almost identical to $\rho_{\M_t}$ for the scalar case and that the optimum is reached when $b_t=c_t$, so that
%
\[
 \widetilde W = \left[\begin{array}{ccc}
1 & a & b_t \\ a & 1 & b_t \\
b_t & b_t & 2d - b_t^2.
\end{array}\right]
\]
We show in Appendix C that in this case no other realization has a better passivity radius. It is also worth pointing out that $\lambda_{\min}W(I,\M_t)$ reaches is optimum value at another value of $t$, but that $\rho_{\M_t}$ is nearly optimal at that point.

\section{Computing the largest value of $\xi^*(X)$} \label{sec:computing}

In this section we describe an algorithm that computes, within a given tolerance
$\tau$, an approximation of the supremum $\Xi$ (see Theorem~\ref{thm:Xi}) of a given minimal realization $\M:=\{A,B,C,D\}$ that is passive.

First of all, if  $\M$ is passive but not strictly passive then $\Xi=0$. If $\M$ is strictly passive, then a
simple upper bound for $\Xi$ follows by the stability bound
\[
\Xi_{up} = 1-\max_j |\lambda_j(A)|.
\]
The procedure to compute $\Xi$ is then to verify for $0 \le \xi \le \Xi_{up}$ the second condition, namely that the pencil \eqref{xipencil} has no unit circle eigenvalues.
The smallest value of $\xi$ in this interval where this condition fails, equals $\Xi$.
(Note that this could be equal to $\Xi_{up}$.) One can then apply a bisection method to this interval and check the presence of unit circle eigenvalues in the given interval.
Setting $\Xi_{lo}=0$, we then have the following procedure.

\medskip
\noindent
{\bf Bisection procedure for computing $\Xi$} \\
$ \xi := (\Xi_{lo} + \Xi_{up})/2, \; \mathrm{\bf if\;} S_\xi \mathrm{\; has \; unit \; circle \; eigenvalues\; \bf then \;} \Xi_{up}:=\xi, \mathrm{\; \bf else\; }  \Xi_{lo}:=\xi.
$

\medskip
Since the interval containing $\Xi$ shrinks by a factor $2$ in each step of this iteration,
in $k=\lceil \log_2(\Xi_{up}/\tau) \rceil$ steps, the interval $[\Xi_{lo},\Xi_{up}]$ will be of length less than or equal to $\tau$.

One can also make use of the computed eigenvalue decompositions to construct an algorithm with faster convergence. For this, we consider the generalized eigenvalue problem
\[
 \Gamma(\xi,\omega):=
\left[ \begin{array}{ccc} 0 & e^{\imath\omega}(\xi-1) I_n+A  & B \\
A^{\mathsf{H}}+ e^{-\imath\omega}(\xi-1)I_n & 0 & C^{\mathsf{H}} \\
B^{\mathsf{H}} & C & D^{\mathsf{H}}+D-\xi I_m \end{array} \right],
\]
which is Hermitian for all real values of $\omega$ and $\xi< 1$.
For a given value of $\hat \xi$ one can check if $\Gamma(\hat \xi,\omega)$ has real eigenvalues $\omega_i$ (they correspond to unit circle eigenvalues of $S_{\hat \xi}$),
and for a given value of $\hat \omega$ one can find the smallest real root $\xi_i$ of $\Gamma(\xi,\hat \omega)$.
These two ideas can be combined in an algorithm for computing $\Xi$ that is very similar to the computation of the $H_\infty$ norm of a transfer function.

We first recall some basic properties of the scalar function $\gamma(\xi,\omega):=\lambda_{\min}\Gamma(\xi,\omega)$, which can be derived from the results
described in \cite{BoyB90} and from the properties of eigenvalues of Hermitian matrices.
\begin{enumerate}
	\item $\gamma(\xi,\omega)$ is a real continuous function of the real variables $\xi$ and $\omega$,
	\item if $\gamma(\hat \xi,\omega) > 0$ for all  $\omega$ then $\hat \xi < \Xi$,
	\item for $\hat \xi < \Xi_{up}$ the real zeros $\omega_k$ of $\gamma(\hat \xi,\omega)$ correspond to a subset of the unit circle eigenvalues $e^{\imath\omega_k}$ of $\Gamma(\hat \xi,\omega)$,
	\item for a given value of $\hat \xi$, $\gamma(\hat \xi,\omega)$ is a quadratic function of $\omega$ in the neighborhood of its local minima,
	\item if $\omega_1 < \omega_2$ are two consecutive zeros of $\gamma(\hat\xi,\omega)$ then at the midpoint $\hat\omega:=(\omega_1+\omega_2)/2$ the smallest real
	root $\widetilde \xi$ of $\Gamma(\xi,\hat\omega)$ lies between 0 and $\hat\xi$ and is an improved upper bound for $\Xi$.
\end{enumerate}

These ideas lead to the following improved algorithm for the computation of $\Xi$.

\medskip

\noindent
{\bf Eigenvalue based  procedure for computing $\Xi$}
\begin{enumerate}
\item $\hat \xi := \Xi_{up}-\tau$;
\item Compute the unit circle eigenvalues $e^{\imath \omega_k}$ of $\Gamma(\hat\xi,\omega)$ and select those corresponding to real zeros $\omega_k$ of $\gamma(\hat \xi,\omega)$;
\item {\bf if} $\gamma(\hat \xi,\omega)$ has no real zeros, {\bf then} $\Xi_{lo}=\hat \xi$, stop;\\
\qquad {\bf else} take the midpoint $\hat\omega:=(\omega_1+\omega_2)/2$ of the largest interval $[\omega_1,\omega_2]$ of these roots \\ \hspace*{.7cm} compute the real roots $\xi_i$ of
$\Gamma(\xi,\hat \omega)$ and update $\Xi_{up} :=\min_i \xi_i$ \\ \hspace*{.7cm} make a guess for $\hat\xi := \Xi_{up}-\tau$ and go to 2.
\end{enumerate}

This algorithm is very similar to the methods proposed in the literature for computing the $H_\infty$ norm of a transfer function (see \eg \cite{BoyB90}) and can therefore
be expected to require only a few iterations to stop with an interval $[\Xi_{lo},\Xi_{up}]$ of size $\tau$.

Note that each step of both algorithms has a complexity that is cubic in the matrix dimensions.
For large scale problems, this complexity becomes a problem, but there are techniques that exploit sparsity to reduce the complexity, see \eg \cite{BeM18,KrV14}.

\section{The distance to passivity} \label{sec:distance}

In this section, we consider the converse problem of computing the smallest perturbation that makes a system passive. Suppose that we are given a minimal system $\M:=\{A,B,C,D\}$ that is not passive. Then we study the problem of computing the smallest perturbation $\Delta_\M$ of the model $\M$ that makes the system
$\M+\Delta_\M$ passive. It is clear that this is equivalent to asking which is the smallest perturbation $\Delta_\M$, measured via the matrix $\Delta_\S$ in \eqref{DeltaS}, such that the LMI $W(X,\M+\Delta_\M) \ge 0$ has a Hermitian and positive semi-definite solution $X$. Moreover $X>0$ if the perturbed system remains minimal.
\begin{definition} The {\em distance to passivity} of a minimal model $\M:=\{A,B,C,D\}$ is the minimum
norm $\|\Delta_\S\|_{2}$ or $\|\Delta_\S\|_{F}$ such that there exists a matrix $X>0$ satisfying
\begin{equation} \label{makepassive}
\widehat W +E\left[\begin{array}{cc} 0 & \Delta_\S \\ \Delta_\S^\mathsf{H} & 0 \end{array}\right]E^\mathsf{T} \ge 0, \;
\mathrm{where} \; \widehat W :=\left[\begin{array}{ccc} X^{-1} &  A & B \\  A^{\mathsf{H}} & X & C^{\mathsf{H}}  \\ B^{\mathsf{H}} &
C & D^{\mathsf{H}} + D \end{array}\right],
\end{equation}
and $E$ is defined in \eqref{defWhatX}.
\end{definition}
Note that (\ref{makepassive}) is an LMI in the parameters of $\Delta_\M$, but it is not linear in $X$.
We will need the following extension of Lemma \ref{inclusion}, for which we consider the LMI for the modified model
$\M_{-\xi} :=\{A_{-\xi},B_{-\xi},C_{-\xi},D_{-\xi}\}:=\{\frac{A}{(1+\xi)},\frac{B}{(1+\xi)},\frac{C}{(1+\xi)},\frac{D+\xi I_m}{(1+\xi)}\}$ with the corresponding transfer function
\begin{equation*} \label{Tminusxi}
T_{-\xi}(z) := C_{-\xi}(zI_n -A_{-\xi})^{-1}B_{-\xi}+D_{-\xi},
\end{equation*}
and corresponding LMI
\begin{equation} \label{shifted2}
\widetilde W(X,\M_{-\xi}) := \left[\begin{array}{ccc} X & XA_{-\xi} & XB_{-\xi} \\
 A_{-\xi}^{\mathsf{H}}X & X & C_{-\xi}^{\mathsf{H}}  \\
B_{-\xi}^{\mathsf{H}}X & C_{-\xi} & D_{-\xi}^{\mathsf{H}}+ D_{-\xi} \end{array}\right]  \ge  0.
\end{equation}
\begin{lemma} \label{inclusion2}
Let $\M:=\{A,B,C,D\}$ be a minimal non-passive system. Then for every $X> 0$ in $\mathbb{H}_n$ there exists a $\xi^*(X)> 0$ such that the LMI \eqref{shifted2} for the system $\M_{-\xi^*(X)}$ holds. Moreover, for every value
$\xi > \xi^*(X)$, the system $\M_{-\xi}$ is passive.
\end{lemma}
\begin{proof}
We have the relation $(1+\xi)\widetilde W(X,\M_{-\xi})=\widetilde W(X,\M) +\xi \diag(X,X,2I_m)$, and since $\widetilde W(X,\M)$ is bounded,
the inequality $\widetilde W(X,\M_{-\xi})\ge 0$ holds for a sufficiently large value of $\xi$. Let $\xi^*(X)$ be the smallest value for which
the passivity condition \eqref{shifted2} holds, then
\[
(1+\xi)\widetilde W(X,\M_{-\xi})=(1+\xi^*(X))\widetilde W(X,\M_{-\xi^*(X)})+(\xi-\xi^*(X)) \diag(X,X,2I_m),
\]
which implies that the passivity condition holds for all $\xi > \xi^*(X)$.
\end{proof}

To determine the distance to passivity, we first restrict ourselves to a perturbation $\Delta_\S$ that has a particular structure.
\begin{theorem} \label{minpass}
The minimum norm perturbation of the type
\begin{equation} \label{restrict}  \S+\Delta_\S= \frac{1}{(1+\xi)} \left( \S + \left[ \begin{array}{cc} 0 & 0 \\ 0 & \xi I_m \end{array} \right]  \right)
\end{equation}
that makes the system $\M$ passive, corresponds to the minimal value of $\xi$ such that the model
$\M_{-\xi}:=\{\frac{A}{(1+\xi)},\frac{B}{(1+\xi)},\frac{C}{(1+\xi)},\frac{D+\xi I_m}{(1+\xi)}\}$
with transfer function $T_{-\xi}(z)$, is passive.
\end{theorem}
\begin{proof}
It follows from \eqref{makepassive} that $\xi$ must satisfy the LMI \eqref{shifted2}
for some $X>0$. By Lemma \ref{inclusion2} there exists a bounded minimal solution, which we call $\Xi$. The model corresponding to $\S+\Delta_\S$ is $\M_{-\xi}$
with transfer function \eqref{Tminusxi}. Therefore $\Xi$ is the
smallest value of $\xi$ that makes the model $\M_{-\xi}$ with transfer function $T_{-\xi}(z)$ become passive. We can then choose $X>0$ from the domain of $\widetilde W(X,\M_{-\Xi})\ge 0$ to satisfy \eqref{shifted2}.
\end{proof}
The minimal value $\Xi$ in Theorem~\ref{minpass} can be computed with the algorithms described in the last section. It thus determines that passivity radius for the constrained class of perturbations \eqref{restrict}.

Since we most likely made some of the eigenvalues of the LMI \eqref{shifted2} strictly positive, rather than nonnegative, we can probably reduce the norm of the perturbation $\Delta_\S$ when removing the constraint \eqref{restrict}. In order to do that, we use a matrix $X$ from the set $\widetilde W(X,\M_{-\Xi})\ge 0$, where $\Xi$ was obtained from the constrained problem. But once $X$ is fixed, condition \eqref{makepassive} beomes an LMI in the unknown perturbation $\Delta_\S$. We can then minimize its 2-norm $\sigma$ by solving the optimization problem
\[
 \min_{\Delta_\S} \sigma, \quad s.t. \quad \left[\begin{array}{cc} \sigma I_{n+m} & \Delta_\S  \\  \Delta_\S^{\mathsf{H}} & \sigma I_{n+m} \end{array}\right] \ge 0, \quad
\widehat W +E\left[\begin{array}{cc} 0 & \Delta_\S \\ \Delta_\S^\mathsf{H} & 0 \end{array}\right]E^\mathsf{T} \ge 0,
\]
or its Frobenius norm $\hat \sigma$ by solving
\[
 \min_{\Delta_\S} \hat \sigma, \quad s.t. \quad \left[\begin{array}{cc} \hat\sigma I_{(n+m)^2} & \mathrm{vec}(\Delta_\S)  \\  \mathrm{vec}(\Delta_\S)^{\mathsf{H}} & \hat\sigma  \end{array}\right] \ge 0, \quad
\widehat W +E\left[\begin{array}{cc} 0 & \Delta_\S \\ \Delta_\S^\mathsf{H} & 0 \end{array}\right]E^\mathsf{T} \ge 0.
\]
Notice that the constrained problem of Theorem \ref{restrict} provided a feasible starting value $\Delta_\S$ for these optimization problems.
We could also use another matrix $X$ that is not in the solution set of $\widetilde W(X,\M_{-\Xi})\ge 0$, but then the norm of the starting point
$\Delta_\S$ constructed from the constrained problem would be larger since  $\xi^*(X)> \Xi$.
\begin{remark}\label{rem4}{\rm
 The same reasoning on how to compute the distance to the nearest passive system can be applied to estimate the distance of a system $x_{k+1}=Ax_k$ that is  unstable to the nearest stable system, see also \cite{GilMS18,GilS17} for the continuous-time case. A result analogous to Theorem~\ref{minpass} would give that a solution of the type
\begin{equation*} A+\Delta_A= A/(1+\xi)
\end{equation*}
%
has a relative error $A^{-1}\Delta_A$ with 2-norm $\frac{\Xi}{(1+\Xi)}$ and Frobenius norm $\frac{\Xi\sqrt{n}}{(1+\Xi)}$,
where $\Xi$ is the minimum value of $\xi$ such that the matrix $A_{-\xi}:=A/(1+\Xi)$ is stable, and this can be used to find an appropriate matrix $X$ for an LMI in $\Delta_A$.
}
\end{remark}

\section{Conclusion}

In this paper we have introduced the notion of normalized passive realizations of a discrete-time system and shown that they share properties with the normalized port-Hamiltonian realizations
of a continuous-time system introduced in \cite{MeV19}. We also showed that the normalized passive realizations typically have a better passivity radius than non-normalized ones.
We have derived methods to maximize a lower bound on the passivity radius and to construct a nearly \emph{optimally robust normalized realization}. The techniques developed in this paper can also
be applied to compute a nearby passive system to a given non-passive one.

\section*{Acknowledgments}
This work was supported by the {\it Deutsche Forschungsgemeinschaft}, through TRR 154 'Mathematical Modelling, Simulation and Optimization using the Example of Gas Networks'. The first author was also supported by {\it the German Federal Ministry of Education and Research BMBF within the project EiFer}.
The second author was also supported  by the Belgian network DYSCO, funded by the Interuniversity Attraction Poles Programme.

\section{Appendix A}
Consider the unimodal optimization problem
\begin{equation}
\label{optgamma}
g(\gamma^*):=\min_{0<\gamma<\infty} g(\gamma), \quad \mathrm{where} \quad g(\gamma):=\| [\;\gamma F_1, \; F_2/\gamma\;]\|_2^2,\quad F_i\in \mathbb{C}^{(2n+m)\times(n+m)}.
\end{equation}
If we define $\alpha:=\|F_1\|_2$ and $\beta:=\|F_2\|_2$ then it was already shown in Theorem \ref{thm:Xpassivity} that $\alpha\beta \le g(\gamma^*) \le 2\alpha\beta$.
We can then derive the following result.
\begin{lemma}
	The infinite search interval for $\gamma$ in the minimization problem \eqref{optgamma} can be replaced by the closed interval $\gamma\in[\gamma_{lo},\gamma_{up}]:=\left [\sqrt{\frac{\beta}{2\alpha}},\sqrt{\frac{2\beta}{\alpha}}\right ]$. Moreover, the function value $g(\gamma^{gm})$ at the geometric mean $\gamma^{gm}:=\sqrt{\frac{\beta}{\alpha}}$ is an upper bound for the minimum.
\end{lemma}
\begin{proof}
	It is easy to see that $g(\gamma)>2\alpha\beta$ outside the interval $\gamma\in[\gamma_{lo},\gamma_{up}]$ and, since $g(\gamma^*)\le 2\alpha\beta$, the minimum must lie in the interval $\gamma\in[\gamma_{lo},\gamma_{up}]$. Any function value in this interval is of course an upper bound for the minimum.
\end{proof}

\section{Appendix B}

In this appendix we describe another characterization of $\rho_\M(X)$. For this, we consider the identity
\[
 M(Q):=\left[\begin{array}{cc} F_1 & F_2 \end{array}\right]
\left[\begin{array}{cc} 0 & Q \\ Q^\mathsf{H} & 0  \end{array}\right]
\left[\begin{array}{cc} F_1^\mathsf{H} \\ F_2^\mathsf{H}  \end{array}\right] = \left[\begin{array}{cc}\gamma F_1 & F_2/\gamma \end{array}\right]
\left[\begin{array}{cc} 0 & Q \\ Q^\mathsf{H} & 0  \end{array}\right]
\left[\begin{array}{cc}\gamma F_1^\mathsf{H} \\ F_2^\mathsf{H} /\gamma \end{array}\right],
\]
which holds for every real $\gamma >0$ and every nonsingular matrix $Q$.
If we constrain $Q$ to be unitary, \ie $QQ^\mathsf{H}=Q^\mathsf{H}Q=I$, then
it follows that
\[
 h(Q):= \|M(Q)\|_2 \le g(\gamma^*):=\min_{0<\gamma<\infty} g(\gamma), \quad g(\gamma):= \sigma_{\max}^2[\;\gamma F_1, \; F_2/\gamma\;] = \| [\;\gamma F_1, \; F_2/\gamma\;]\|_2^2.
 \]
We now prove that we also have
\begin{equation} \label{duality}
 g(\gamma^*)= h(Q^*) := \max_{QQ^\mathsf{H}=Q^\mathsf{H}Q=I} h(Q) = \max_{QQ^\mathsf{H}=Q^\mathsf{H}Q=I} \|M(Q)\|_2,
\end{equation}
which we prove by constructing a matrix $Q$ so that (\ref{duality}) holds.
It follows from Theorem~\ref{thm:mingamma} that the minimizing right singular vector $z:=\left[\begin{array}{cc} u \\ v  \end{array}\right]$ satisfies
\[ \left[\begin{array}{cc}\gamma F_1 & F_2/\gamma \end{array}\right]
\left[\begin{array}{cc} u \\ v  \end{array}\right] = \sigma_{\max} w,
\quad
\left[\begin{array}{cc}\gamma F_1^\mathsf{H} \\ F_2^\mathsf{H} /\gamma \end{array}\right]w =\sigma_{\max} \left[\begin{array}{cc} u \\ v  \end{array}\right], \quad \|u\|_2=\|v\|_2.
\]
It is then easy to verify then that for a unitary  $Q$ satisfying $Qv=u$ and $Q^\mathsf{H}u=v$, then $M(Q)z=\sigma_{\max}^2z$.

We can now use this construction to show that normalized realizations have a better passivity radius than non-normalized ones. Let $\widehat W=R^\mathsf{H}R$ be the Cholesky factorization of an arbitrary model $\M$.
The Cholesky factorization of the corresponding matrix
\[
\widehat W_n= \diag(T^{-1},T^\mathsf{H},I_m)\widehat W \diag(T^{-\mathsf{H}},T,I_m)
\]
of the normalized model $\M_T:=\{T^{-1}AT,T^{-1}B,CT,D\}$ is then given by
\[
\widehat W_n:=R_n^\mathsf{H}R_n = \diag(T^{-1},T^\mathsf{H},I_m) R^\mathsf{H}R \diag(T^{-\mathsf{H}},T,I_m)
\]
 and the relation $R^{-\mathsf{H}}=R_n^{-\mathsf{H}}
\diag(T^{-1},T^{\mathsf{H}},I_m)$ then yields
\begin{eqnarray*}
F_1&:=&R^{-\mathsf{H}}E_1=R_n^{-\mathsf{H}}E_1\diag(T^{-1},I_m)= N_1\diag(T^{-1},I_m),\\
F_2&:=&R^{-\mathsf{H}}E_2=R_n^{-\mathsf{H}}E_2\diag(T^{\mathsf{H}},I_m)=N_2\diag(T^{\mathsf{H}},I_m).
\end{eqnarray*}
It then follows from \eqref{duality} that
\begin{eqnarray*}
 \rho^{-1}_\M(X) & = & \min_{0<\gamma<\infty} \|\left[\begin{array}{cc}\gamma F_1 & F_2/\gamma \end{array}\right]
\left[\begin{array}{cc}\gamma F_1^\mathsf{H} \\ F_2^\mathsf{H} /\gamma \end{array}\right]\|_2 \\
& \ge &
\| \left[\begin{array}{cc}\gamma F_1 & F_2/\gamma \end{array}\right]
\left[\begin{array}{cc} 0 & I_n \\ I_n & 0  \end{array}\right]
\left[\begin{array}{cc}\gamma F_1^\mathsf{H} \\ F_2^\mathsf{H} /\gamma \end{array}\right]\|_2 \\
 & = & \| \left[\begin{array}{cc} N_1 & N_2 \end{array}\right]
 \left[\begin{array}{cc} 0 & I_n \\ I_n & 0  \end{array}\right]
 \left[\begin{array}{cc} N_1^\mathsf{H} \\ N_2^\mathsf{H} \end{array}\right]\|_2 \\
 & = & \|N_1N_2^\mathsf{H}+N_2N_1^\mathsf{H}\|_2.
\end{eqnarray*}
Note that
\[
h(I)=\|N_1N_2^\mathsf{H}+N_2N_1^\mathsf{H}\|_2 = \|  R^{-\mathsf{H}}\left[\begin{array}{ccc} 0 & I_n & 0 \\ I_n & 0 & 0 \\ 0 & 0 & 2I_m \end{array}\right] R^{-1} \|_2 \le 2\|N_1\|_2\|N_2\|_2
\]
 but we need a lower bound for $\rho^{-1}_{\M_T}(X)$.
If $\widehat W_n$ commutes with $J:=\left[\begin{array}{ccc} 0 & I_n & 0 \\ I_n & 0 & 0 \\ 0 & 0 & I_m \end{array}\right]$, which implies that $\left[\begin{array}{cc} A_T & B_T \end{array}\right]= \left[\begin{array}{cc} A_T^\mathsf{H} & C_T^\mathsf{H} \end{array}\right]$ and hence that $\M_T$
is its own dual system, then
\[
\left \| ( R^{-\mathsf{H}}\left[\begin{array}{ccc} 0 & I_n & 0 \\ I_n & 0 & 0 \\ 0 & 0 & 2I_m \end{array}\right] R^{-1} )^2 \right \|_2 = \left \| ( R^{-\mathsf{H}}\left[\begin{array}{ccc} I_n & 0 & 0 \\ 0 & I_n & 0 \\ 0 & 0 & 2I_m \end{array}\right] R^{-1})^2 \right \|_2 = \|(D_s^{-1}\widehat W_n^{-1} D_s^{-1})^2\|_2.
\]
In this case it follows that
\[
  \rho_{\M}(X) \le \lambda_{\min}(D_s\widehat W_n D_s) \le \rho_{\M_T}(X),
\]
which implies that such  a normalized realization has a better passivity radius than the corresponding non-normalized realization.

\end{document}